\documentclass[a4paper,12pt]{article}
\usepackage[utf8]{inputenc}
\usepackage{amssymb}
\usepackage{amsmath}
\usepackage{amsthm}
\usepackage{tikz}
\usepackage{hyperref}
\usetikzlibrary{arrows,calc,decorations.pathreplacing,fadings,intersections}
\usetikzlibrary{external}

\newtheorem{theorem}{Theorem}
\newtheorem{proposition}{Proposition}
\newtheorem{lemma}[proposition]{Lemma}
\newtheorem{corollary}[proposition]{Corollary}

\theoremstyle{definition}
\newtheorem{remark}{Remark}

\newcommand{\IE}{\mathbb{E}}
\newcommand{\IN}{\mathbb{N}}
\newcommand{\IP}{\mathbb{P}}
\newcommand{\IR}{\mathbb{R}}

\newcommand{\dd}{{\rm d}}
\newcommand{\dist}{\mathrm{dist}}
\newcommand{\diam}{\mathrm{diam}}

\newcommand{\vol}{\mathrm{vol}}

\title{Random approximation of convex bodies in\\ Hausdorff metric}
\author{Joscha Prochno\footnote{Faculty of Computer Science and Mathematics, University of Passau, Dr.-Hans-Kapfinger-Str. 30, 94032 Passau, Germany, {\tt joscha.prochno@uni-passau.de, mathias.sonnleitner@uni-passau.de}},
Carsten Sch\"utt\footnote{Department of Mathematics, University of Kiel, Heinrich-Hecht-Platz 6, 24118 Kiel, Germany, 
{\tt schuett@math.uni-kiel.de }},
Mathias Sonnleitner$^{*}$,
Elisabeth M. Werner\footnote{Department of Mathematics, Case Western Reserve University, 2145 Adalbert Road, Cleveland, OH 44106, USA,
{\tt elisabeth.werner@case.edu}}
}
\begin{document}

\maketitle

\abstract
While there is extensive literature on approximation, deterministic as well as random,  of general convex bodies $K$ in the symmetric difference metric, or other metrics arising from intrinsic volumes, very little is known for corresponding random results in the Hausdorff distance when the approximant $K_n$ is given by the convex hull of $n$ independent random points chosen uniformly on the boundary or in the interior of $K$. When $K$ is a polygon and the points are chosen on its boundary, we determine the exact limiting behavior of the expected Hausdorff distance between a polygon as $n\to\infty$. From this we derive the behavior of the asymptotic constant for a regular polygon in the number of vertices.

\section{Introduction}

A convex body $K$ in $\mathbb R^{d}$ can be approximated arbitrarily well by polytopes in  the symmetric difference metric or the Hausdorff metric. This can be achieved in many ways. For instance,
 by taking  the convex hull of sufficiently many random points sampled inside the convex body $K$ or on its boundary $\partial K$. In this paper we investigate  how much a convex body $K$ differs from the convex hull $K_{n}$ of $n\in\mathbb N$ points chosen at random from $K$. More specifically, we are interested in the order of magnitude of the following two quantities
$$
\mathbb P(\delta_{H}(K,K_{n})\geq t), \quad t>0,
\hskip 20mm\mbox{and}\hskip 20mm
\mathbb E[\delta_{H}(K,K_{n})],
$$
where the points spanning the polytope $K_n$ are chosen uniformly at random either from $K$ or from the boundary of $K$, with respect to the normalized Hausdorff measure of appropriate dimension. We refer to the end of this section for the definitions of the relevant notions.

There is a vast  amount  of literature on these problems when  the error of approximation is measured by the expectation of the symmetric difference metric $d_S$: for points chosen uniformly at random  inside the body, A. R\'{e}nyi and  R. Sulanke studied this question in dimension $2$ \cite{RS63,RS64};  
I. B\'ar\'any proved asymptotics  involving  the Gau\ss-Kronecker curvature for smooth convex bodies in arbitrary dimension \cite{Bar92};  C. Sch\"utt  generalized this  to arbitrary convex bodies in \cite{Sch94}. In \cite{BB93}, I. B\'ar\'any and Ch. Buchta determined the asymptotic behavior of the expected symmetric difference metric for convex polytopes in arbitrary dimensions. 

More recently, M. Reitzner \cite{Rei02} and C. Sch\"utt and E. M. Werner \cite{SW03} proved asymptotics for $\IE[ d_S(K,K_n)]$ for random points chosen from the boundary of a sufficiently smooth convex body. In \cite{RSW23} these three authors extended these results to the case when the convex bodies are polytopes. 

The extensive literature on approximation of general convex bodies, deterministic as well as random,  mostly involves  the symmetric difference metric, or other metrics coming from intrinsic volumes. For smooth convex bodies the asymptotic constants are often given as integrals of curvature(s) of $K$ over its boundary and for polytopes combinatorial quantities appear. For more information, we refer to the survey \cite{PSW22}. 
However, very little is known for corresponding results in the Hausdorff distance and, for the reader's convenience, we provide a short overview.

S. Glasauer and R. Schneider \cite{GS96} considered the case of a convex body with $C^{3}_{+}$-boundary, that is three times continuously  differentiable boundary with everywhere positive curvature, and a probability measure $\mathbb P$ with  density $h>0$ of class $C^{1}$ on the boundary. For random points sampled according to this density they proved the asymptotics
\begin{equation}
	\label{eq:gs-smooth}
\hskip -5mm
\left(\frac{n}{\log n}\right)^{\frac{2}{d-1}}\delta_{H}(K,K_{n})
\xrightarrow[n\to\infty]{\mathbb{P}}
\frac{1}{2}\left(
\frac{1}{\operatorname{vol}_{d-1}(B_{2}^{d-1})}\max_{x\in\partial K}\frac{\sqrt{\kappa_{K}(x)}}{h(x)}
\right)^{\frac{2}{d-1}},
\end{equation}
where $\kappa_K$ is the Gau\ss-Kronecker curvature, $B_2^{d-1}$ the $(d-1)$-dimensional unit ball and $ \xrightarrow[n\to\infty]{\mathbb{P}}$ denotes convergence in probability as $n\to\infty$. The case $d=2$ was already obtained by R. Schneider \cite{Sch88}.

L. D\"umbgen and G. Walther \cite{DW96} showed upper bounds for the Hausdorff distance. More precisely, they proved that if $K_n$ is the convex hull of random points uniformly distributed on the boundary $\partial K$ of an arbitrary convex body $K$, then with probability one
 \[
\delta_{H}(K,K_{n})
\le C \Big(\frac{\log n}{n}\Big)^{\alpha}
 \]
 for some constant $C>0$ and all $n\in\mathbb N$, where $\alpha=\frac{1}{d-1}$.  The exponent $\frac{1}{d-1}$ can be improved to $\frac{2}{d-1}$ if the body is smooth in the following sense: at each boundary point, there is a unique normal and this normal is Lipschitz-continuous  on $\partial K$. The latter includes the case of $C^3_+$-boundary, see, e.g., G. Walther \cite[Thm. 1]{Wal99}. In \cite{DW96},  a result for points sampled uniformly inside $K$ was proven, where $\alpha=\frac{1}{d}$ in general and $\alpha=\frac{2}{d+1}$ for smooth bodies. 

A direct consequence of the convergence in \eqref{eq:gs-smooth} and the bounds in \cite{DW96} is the following, 
which we only state for the case of the uniform distribution on $\partial K$. We prove it in Section \ref{Beweise}.

 \begin{proposition}\label{pro:smooth}
Let $K$ be of class $C^3_+$ and, for $n\in\IN$, let $K_n$ be the convex hull of $n$ independent random points uniformly distributed on $\partial K$. Then
\[
\lim_{n\to\infty}\left(\frac{n}{\log n}\right)^{\frac{2}{d-1}}\IE[ \delta_{H}(K,K_{n})]
= c_K,
\]
where $c_K\in(0,\infty)$ is the limiting constant in \eqref{eq:gs-smooth} for $h\equiv \vol_{d-1}(\partial K)^{-1}$.
\end{proposition}

In the case when  points are chosen independently and uniformly at random within a convex body $K\subset \IR^d$, we are only aware of exact asymptotic results in dimension two. However, for higher-dimensional smooth convex bodies there is  work in progress due to P. Calka and J. Yukich. H. Br\"aker, T. Hsing and N.H. Bingham \cite{BHB98} obtained a formula for the limiting distribution function for smooth convex bodies $K$ and convex polytopes in $\mathbb R^{2}$. More precisely, let $Q$ be a convex polytope in $\mathbb R^{2}$ with interior angles $\alpha_{1},\dots,\alpha_{M}$, and let $Q_{n}$ be the convex hull of $n$ points chosen uniformly at random from $Q$. Then \cite[Thm. 3]{BHB98} states that
$$
\lim_{n\to\infty}\mathbb P(\sqrt{n}\delta_{H}(Q,Q_{n})\leq t)
=\prod_{i=1}^{M}(1-p_{i}(t)),\qquad t> 0,
$$
where for $i\in\{1,\dots,M\}$
\begin{equation}
 \label{eq:bhb-integral}
p_{i}(t)
=\left\{
\begin{array}{cc}
\int_{0}^{\alpha_{i}}h_{i}(t,\alpha)\dd\alpha+\exp\left(-\frac{t^{2}}{2\operatorname{vol}_{2}(Q)}\tan\alpha_{i}\right) 
& 0<\alpha_{i}<\frac{\pi}{2}
\\
\int_{\alpha_{i}-\frac{\pi}{2}}^{\frac{\pi}{2}}h_{i}(t,\alpha)\dd\alpha
& \frac{\pi}{2}\leq \alpha_{i}<\pi,
\end{array}
\right.
\end{equation}
with
$$
h_{i}(t,\alpha)
=\exp\left(-\frac{t^{2}}{2\operatorname{vol}_{2}(Q)}\left(\tan\alpha+\tan(\alpha_{i}-\alpha)\right)\right)
\frac{t^{2}}{2\operatorname{vol}_{2}(Q)}\tan^{2}\alpha.
$$
Note that the limiting distribution is a product of distribution functions, one for each vertex. Below, we use tail bounds to show that the convergence in distribution extends also to convergence in expectation. Let us note that V.E. Brunel \cite{Bru19} provides concentration inequalities for random polytopes in Hausdorff distance, but an additional log-factor appears in the resulting bounds on the expectation.

The only result we could find in the literature on the behavior of $\IE [\delta_H(K,K_n)]$ when $K$ is a polytope  is due to I. B\'ar\'any \cite{Bar89}, who states that it is ``almost trivial" that 
$$
 \IE [\delta_H(K,K_n)]\asymp n^{-1/d},
$$
where $\asymp$ denotes the equivalence of the two expressions up to positive absolute constants. Here, the $n$ points are sampled uniformly inside the polytope. The case of sampling the points uniformly from the boundary is very similar and leads to the following asymptotics;  for completeness a proof will be presented below.

\begin{proposition}\label{pro:barany}
Let $Q\subset\IR^d$ be a convex polytope and, for each $n\in\IN$, let $Q_n$ be the convex hull of $n$ points sampled independently and uniformly from the boundary of $Q$. Then
\[
\IE [\delta_H(Q,Q_n)]\asymp n^{-1/(d-1)}.
\]
\end{proposition}

I. B\'ar\'any's result gives the dependence on the number of chosen points but nothing is known about the precise asymptotic behavior. 

In this paper we consider a polygon $Q$ in $\mathbb{R}^2$ and determine the exact constants in $\IE[ \delta_H(Q,Q_n)]$ as the number of points chosen on the boundary $\partial Q$ tends to infinity. We will mainly follow the approach of \cite{BHB98}, which is based on Poisson processes, and provide the necessary arguments to achieve convergence in expectation. 

We now present the main results of this paper. We start with the asymptotic behavior of the expected Hausdorff distance of a convex polygon and the random polytope formed by the convex hull of $n$ independent uniform points on the boundary of this polygon as $n\to\infty$.

\begin{theorem} \label{thm:polygon}
Let $Q$ be a convex polygon with $M\ge 3$ vertices, interior angles $\alpha_1,\dots,\alpha_M\in (0,\pi)$ and total edge length  $\left| \partial \, Q \right|$. For each $n\in\IN_{\geq 3}$, let $X_1,\dots,X_n$ be  independent and uniformly distributed points on the boundary of $Q$, and $Q_n=[X_1,\dots,X_n]$ their convex hull. Then
\begin{equation}\label{polygon-1}
\lim_{n\to\infty}n \, \IE[\delta_H(Q,Q_n)]
= \left| \partial \, Q \right| \int_{0}^{\infty}1-\prod_{i=1}^M\Big(1-r\int_{\ell_{\alpha_i}}^{\infty}e^{-r(y+h(y,\alpha_i))}\dd y \Big)\dd r.
\end{equation}
Here,
\begin{equation}\label{polygon-3}
\ell_{\alpha}:=
\begin{cases}
1 &:\,\alpha<\pi/2,\\
\sin(\alpha)^{-1}&:\,\alpha\ge \pi/2,
\end{cases}
\end{equation}
and
\begin{equation}\label{polygon-4}
h(y,\alpha)
:= y\frac{\sin(\alpha)\sqrt{y^2-1}-\cos(\alpha)}{\sin^2(\alpha)y^2-1}
\end{equation}
if
\[
\alpha<\pi/2 \text{ and }y\in [1,\cos(\alpha)^{-1}] \quad\text{or}\quad \alpha\ge \pi/2 \text{ and } y>\sin(\alpha)^{-1}.
\]
If  $\alpha<\pi/2$ and $y>\cos(\alpha)^{-1}$, then we set $h(y,\alpha):=1$.	
\end{theorem}

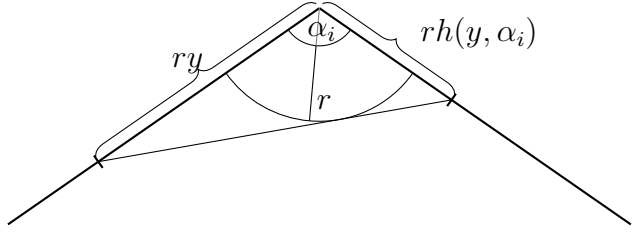
\begin{figure}[ht]
	\begin{center}
		\begin{tikzpicture}[scale=.5,rotate=-35]
			\draw[white,name path=secondedge] (0,0) -- (10,0);
			\draw[thick] (0,0) -- (10,0);
			\draw[white, name path=firstedge] (0,0) -- ({10*cos(250)},{10*sin(250)});
			\draw[thick] (0,0) -- ({10*cos(250)},{10*sin(250)});
			\draw (1,0) arc[start angle=360, end angle=250, radius=1];
			\draw (3,0) arc[start angle=360, end angle=250, radius=3];
			\draw[black!50] ({0*cos(300)},{0*sin(300)}) -- ({1*cos(300)},{1*sin(300)});
			\draw ({1*cos(300)},{1*sin(300)}) -- ({3*cos(300)},{3*sin(300)}) node[near end,xshift=5pt]{$r$};
			\node at (0,0) [xshift=1pt,yshift=-9pt] {$\alpha_i$};

			\begin{scope}
				\clip (0,0)--({10*cos(250)},{10*sin(250)})--(10,0)--(0,0);
				\draw[name path=someline] ({3*cos(315)	+10*cos(225)},{3*sin(315)+10*sin(225)}) -- ({3*cos(315)	-10*cos(225)},{3*sin(315)-10*sin(225)});
				\fill [name intersections={of=someline and firstedge, by=t}] (t) circle (.001);
				\fill [name intersections={of=someline and secondedge, by=s}] (s) circle (.001);
			\end{scope}
			\draw[thick] ($(t)-({.2*cos(340)},{.2*sin(340)})$) --($(t)+({.2*cos(340)},{.2*sin(340)})$);
			\draw[decorate,decoration={brace,amplitude=5pt,raise=2pt,mirror}] (0,0) -- ($(t)$) node [midway,xshift=-8pt,yshift=8pt]{$ry$};
			\draw[decorate,decoration={brace,amplitude=5pt,raise=2pt}] (0,0) -- ($(s)$) node [midway,xshift=35pt,yshift=8pt]{$rh(y,\alpha_{i})$};
			\draw[thick] ($(s)-({.2*cos(90)},{.2*sin(90)})$) --($(s)+({.2*cos(90)},{.2*sin(90)})$);
		\end{tikzpicture}
	\end{center}
	\caption{
	The exponent $r \,h(y,\alpha_i)$ in Theorem~\ref{thm:polygon} is the length  cut off from the boundary $\partial Q$ by a tangent line to a circle with radius $r$  centered at a vertex  that intersects one edge at distance $ry$ from the vertex and the other necessarily at distance $r h(y,\alpha_i)$. For $\alpha_i<\pi/2$ and $y>\cos(\alpha_i)^{-1}$ the picture is different, see  Figure~\ref{fig:small-alpha} below.
	}
\end{figure}
	\tikzsetnextfilename{smallalpha}
	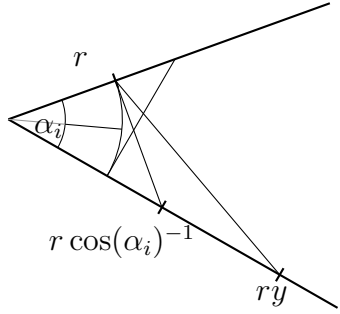
\begin{figure}[h]
		\begin{center}
			\begin{tikzpicture}[scale=.5,rotate=20]
				\def\alph{310}
				\def\line1{360}
				\draw[white,name path=secondedge] (0,0) -- (13,0);
				\draw[thick] (0,0) -- (9,0);
				\draw[white, name path=firstedge] (0,0) -- ({14*cos(\alph)},{14*sin(\alph)});
				\draw[thick] (0,0) -- ({10*cos(\alph)},{10*sin(\alph)});
				\draw (1.5,0) arc[start angle=360, end angle=\alph, radius=1.5];
				\draw (3,0) arc[start angle=360, end angle=\alph, radius=3];
				\draw[color=black!40] (0,0)--({1.5*cos(180+\alph/2)},{1.5*sin(180+\alph/2)});
				\draw ({1.5*cos(180+\alph/2)},{1.5*sin(180+\alph/2)}) -- ({3*cos(180+\alph/2)},{3*sin(180+\alph/2)});% node[near end,yshift=-5pt]{$r/n$};
				\node at (0.2,-0.7) [xshift=9pt,yshift=5pt]{$\alpha_{i}$};

				\begin{scope}
					\clip (0,0)--({13*cos(\alph)},{13*sin(\alph)})--(13,0)--(0,0);
					\draw[name path=rcosline] ({3*cos(\line1)	+10*cos(\line1-90)},{3*sin(\line1)+10*sin(\line1-90)}) -- ({3*cos(\line1)	-10*cos(\line1-90)},{3*sin(\line1)-10*sin(\line1-90)});
					\draw[name path=someline] ({3*cos(360)	+10*cos(380-90)},{3*sin(360)+10*sin(380-90)}) -- ({3*cos(360)	-10*cos(380-90)},{3*sin(360)-10*sin(380-90)});
					\draw[name path=rline] ({3*cos(\alph)	+10*cos(\alph+90)},{3*sin(\alph)+10*sin(\alph+90)}) -- ({3*cos(\alph)	-10*cos(\alph+90)},{3*sin(\alph)-10*sin(\alph+90)});
					\fill [name intersections={of=rcosline and firstedge, by=rcos}] (rcos) circle (.001);
					\fill [name intersections={of=someline and firstedge, by=t}] (t) circle (.001);
					\fill [name intersections={of=someline and secondedge, by=s}] (s) circle (.001);
					\fill [name intersections={of=rline and firstedge, by=r0}] (r0) circle (.001);
					\fill [name intersections={of=rline and secondedge, by=rcos2}] (rcos2) circle (.001);
				\end{scope}
				\draw[thick] ($(t)-({.2*cos(\alph+90)},{.2*sin(\alph+90)})$) --($(t)+({.2*cos(\alph+90)},{.2*sin(\alph+90)})$);
				\node at (t)[xshift=-3pt,yshift=-8pt] {$ry$};
				\draw[thick] ($(s)-({.2*cos(90)},{.2*sin(90)})$) --($(s)+({.2*cos(90)},{.2*sin(90)})$);
				\node at (s)[xshift=-13pt,yshift=8pt] {$r$};
				\draw[thick] ($(rcos)-({.2*cos(\alph+90)},{.2*sin(\alph+90)})$) --($(rcos)+({.2*cos(\alph+90)},{.2*sin(\alph+90)})$);
				\node at (rcos)[xshift=-15pt,yshift=-13pt] {$r\cos(\alpha_{i})^{-1}$};
			\end{tikzpicture}
		\end{center}
		\caption{In the case $\alpha_i\in (0,\pi/2)$, tangents to the circle with radius $r$ centered at a vertex which intersect one edge at distance greater than $r\cos(\alpha_i)^{-1}$ from the vertex intersect the other edge at distance $r$.}
		\label{fig:small-alpha}
	\end{figure}

In order to determine the dependence on the interior angles of the polygon, we deduce the following asymptotics from Theorem~\ref{thm:polygon} in the case when all angles are greater than $\frac{\pi}{2}$.

\begin{corollary}\label{cor:asymptotics}
If $Q$ and $Q_n$ are as in Theorem~\ref{thm:polygon} (i.e., $M\geq 3$) and $\alpha_1,\dots,\alpha_M\ge \pi/2$, we have
\begin{equation*}\label{polygon-5}
\frac{1}{5} A(Q) 
\le \lim_{n\to\infty}n \, \IE[\delta_H(Q,Q_n)]
\le A(Q),
\end{equation*}
where 
\[
A(Q)
:= \left| \partial \, Q \right|\sum_{k=1}^{M}(-1)^{k+1}\sum_{1\leq i_{1}<\cdots< i_{k}\leq M} \frac{1}{\frac{1}{\sin\alpha_{i_{1}}}+\cdots+\frac{1}{\sin\alpha_{i_{k}}}}.
\]
\end{corollary}

In the case when all angles are equal, $Q$ is a regular polygon and we obtain the following asymptotics.

\begin{corollary}\label{cor:mgon}
Let $Q_M^{\rm reg}$ be a regular polygon with $M\ge 4$ vertices and $Q_{M,n}$ the convex hull of $n$ points independently and uniformly sampled from the boundary of $Q_M^{\rm reg}$. Then
\[
\lim_{n\to\infty}n \, \IE[\delta_H(Q_M^{\rm reg},Q_{M,n})]\asymp \left |\partial Q_M^{\rm reg}\right| \, \frac{\log M}{M}. 
\]
\end{corollary}
\begin{proof}
Let $\beta$ be the exterior angle of $Q_M^{\rm reg}$. Then $\beta=\frac{2 \pi}{M}$ and thus $\alpha= \pi-\beta = \pi(1-\frac{2}{M}) \geq \frac{\pi}{2}$, as $M\geq 4$.  Hence  we 
can apply Corollary~\ref{cor:asymptotics}.
 Corollary~\ref{cor:asymptotics} yields for a regular $M$-gon that $\lim_{n\to\infty}n \, \IE[\delta_H(Q,Q_n)]$ is up to absolute constants equal to
\begin{eqnarray*}
\sum_{k=1}^{M}(-1)^{k+1}\sum_{1\leq i_{1},\cdots i_{k}\leq M}
\frac{1}{\frac{1}{\sin\alpha_{i_{1}}}+\cdots+\frac{1}{\sin\alpha_{i_{k}}}}
=\sum_{k=1}^{M}(-1)^{k+1}{M\choose k}\frac{\sin\alpha_{M}}{k},
\end{eqnarray*}
where
\[
\sin(\alpha_M)=\sin(\pi(1-2/M))\asymp \frac{2\pi}{M}.
\]
By formula \cite[0.155/4]{GR07}, it holds that
\[
\sum_{k=1}^{M}\binom{M}{k}(-1)^{k+1}\frac{1}{k}
= \sum_{k=1}^{M}\frac{1}{k}
\asymp \log M,
\]
which completes the proof.
\end{proof}

We conclude this section with  a remark on the log-factor in Corollary~\ref{cor:mgon} and a remark on higher moments.

\begin{remark}
	The logarithm appearing in Corollary~\ref{cor:mgon} can be heuristically derived from the asymptotic independence of the random variables
\[
Z_i:=n\, \dist(v_i,K_n),\quad i\in\{1,\dots,M\},
\]
which is crucial in the proof of Theorem~\ref{thm:polygon} below. Assuming that these random variables are in fact independent, we can use Lemma~\ref{lem:hd-vertex} below to write
\[
n\,\IE[ \delta_H(Q_M^{\rm reg},Q_{M,n})]
=\IE \big[\max_{1\le i\le M} Z_i\big]
\]%(see, e.g., \cite[Prop. 3.5.8.]{AGM15} and \cite[Ex. 2.2.7]{Tal14})
as a maximum of independent random variables. Bounds appearing in the proof of Corollary~\ref{cor:asymptotics} allow us to conclude that the tails $\IP(M\cdot Z_i\ge r)$ decay exponentially in $r$. This yields that $M\IE\big( \max_{1\le i\le M}Z_i\big)$ is of order $\log M$, which is the statement of Corollary~\ref{cor:mgon}. To give an idea for this implication, note that by a substitution we have
\begin{align*}
M\, \IE \Big[\max_{1\le i\le M}Z_i\Big]
& =\int_0^{\infty}\IP\Big(\max_{1\le i\le M}Z_i\ge r/M\Big)\dd r \cr
& =(\log M)\int_0^{\infty}1-\prod_{i=1}^M\big(1-\IP(M\cdot Z_i\ge x\log M)\big)\dd x,
\end{align*}
and we can insert the tail bounds to derive the stated asymptotics.
\end{remark}
\vskip 3mm
\begin{remark}
	We can also establish asymptotics for arbitrary moments. Let  $\gamma>0$ be a real  number.  Modifying the first steps of the proof of Theorem \ref{thm:polygon} and using that 
	\[
	\IE[\delta_H(Q,Q_n)^{\gamma}]
	=\gamma \int_0^{\infty}r^{\gamma-1}\IP(\delta_H(Q,Q_n)\ge r)\dd r
	\]
	will give, with the notation of Theorem~\ref{thm:polygon}, that
	\[
	\lim_{n\to\infty}n^{\gamma}\IE[\delta_H(Q,Q_n)^{\gamma}]
	= \gamma \,  \left|\partial \, Q\right|^{\gamma}\int_{0}^{\infty}r^{\gamma-1}\Bigg(1-\prod_{i=1}^M\Big(1-r\int_{\ell_{\alpha_i}}^{\infty}e^{-r(y+h(y,\alpha_i))}\dd y\Big)\Bigg)\dd r.
	\]
\end{remark}

\vskip 4mm

\noindent
We conclude this section by introducing  notation and relevant notions. 
We then present the proofs of our results in Section \ref{Beweise}.

\medskip

% % % % % % % % % % % % % % % % % % % %
\textbf{Notation. } The Euclidean ball with center $x\in\IR^d$ and radius $\rho>0$ is denoted by $B_{2}^{d}(x,\rho)$. We write  in short  $B_{2}^{d}:=B_{2}^{d}(0,1)$ for the unit ball.  For a set $A$,  its $d$-dimensional Hausdorff measure or volume is  $\vol_d(A)$ or $|A|$, its diameter is  $\diam(A)$.  The distance of the set $A$ to another set $B$ is 
$\dist(A,B) = \inf\{\|x-y\|:  x \in A, y \in B\}$. In particular, $\dist(x,K)=\inf_{y\in K}\|x-y\|$. 
\par
The symmetric difference metric between two convex bodies $C$ and $K$, i.e., compact, convex subsets of $\mathbb R^{d}$ with nonempty interior, is
$$
d_{S}(C,K)=\operatorname{vol}_{d}(C\setminus K)+\operatorname{vol}_{d}(K\setminus C)
$$
and the Hausdorff metric between two convex bodies $C$ and $K$ is
$$
\delta_{H}(C,K)=\inf\{\rho\ge 0\colon  C\subseteq K+B_{2}^{d}(0,\rho)\hskip 1mm\mbox{and}\hskip 1mm K\subseteq C+B_{2}^{d}(0,\rho)\},
$$
where $ C+K=\{x+y\colon x\in C \hskip 1mm\mbox{and}\hskip 1mm y\in K\} $ is their Minkowski sum. The convex hull of points $x_{1},\dots,x_{n}\in\mathbb R^{d}$ is a convex polytope denoted by $[x_{1},\dots,x_{n}]$. %To sample  points uniformly inside of a convex body $K$ or on its boundary $\partial K$ we use the normalized restriction of the Hausdorff measure of the appropriate dimension.

% % % % % % % % % % % % % % % % % % % % %
% % % % % % % % % % % % % % % % % % % % %

\section{Proofs}\label{Beweise}

We first give the proof of Theorem~\ref{thm:polygon}, deduce Corollary~\ref{cor:asymptotics}, and then prove Propositions~\ref{pro:smooth} and~\ref{pro:barany}.
To do so, we need additional  lemmas.

An important observation is  the fact that for a convex polytope the Hausdorff distance to a convex subset is attained at the vertices and asymptotically, as $n\to\infty$, ``everything happens at the vertices''. 

\begin{lemma} \label{lem:hd-vertex}
	Let $d\ge 2$. Let $M \in \mathbb{N}$ and let  $Q\subset\IR^d$ be a convex polytope with vertices $v_{1},\dots,v_M$. If  $K\subset Q$ is a convex,  nonempty subset, then 
	\[
	\delta_H(Q,K)
	=\max_{1\le i\le M} \dist(v_i,K).
	\]
\end{lemma}
This lemma is known, see, e.g., \cite[Lem. 2]{Bru19}, and follows from the fact that the convex continuous function $\dist(\cdot,K)$ attains its maximum on $Q$ at a vertex. 

The following lemma will be useful for geometric computations.
\begin{lemma}\label{lem:triangle}
Let $D$ be a triangle in $\mathbb R^{2}$ with sides of lengths $a,b,c$ and  angle $\alpha\in (0,\pi)$ between the sides $b$ and $c$. Then the height $h$ with respect to the side $a$ equals
\begin{equation}\label{eq:triangle1}
h=\frac{bc\sin\alpha}{\sqrt{b^2+c^2-2bc\cos\alpha}}.
\end{equation}
\end{lemma}

\begin{proof}
Formula \eqref{eq:triangle1} follows from elementary trigonometry.  The  area $|D|$ of $D$ equals
\[
\frac{1}{2}a\cdot h=
|D|=\frac{1}{2}bc\sin\alpha.
\]
Using the law of cosine, $ a^{2}=b^{2}+c^{2}-2bc\cos\alpha$, we obtain
\[
h=\frac{bc}{a}\sin\alpha=\frac{bc\sin\alpha}{\sqrt{b^{2}+c^{2}-2bc\cos\alpha}},
\]
which proves the desired formula.
\end{proof}

As a consequence of Lemma~\ref{lem:triangle} we deduce the function $h$ in Theorem~\ref{thm:polygon}. %We need one more geometric preparation, which essentially allows us to conclude the probability that the convex hull of random points on the boundary of a cone has distance $r>0$ to the apex given the distance of to the closest point on one side to the origin.

\begin{lemma} \label{lem:tangent}
Let $C$ be the cone in $\mathbb R^{2}$ spanned by vectors $u,v\in\IR^2$ of unit length and enclosing an angle $\alpha\in (0,\pi)$. For any $r>0$, the tangent line to $B_2^2(0,r)\cap C$ which intersects one edge in $tu$, where $t\ge r$, passes through $r h(t/r,\alpha)v$, where $h$ is as in Theorem~\ref{thm:polygon}.
\end{lemma}
\begin{proof}
	Let $F_u=\{tu\colon t\ge 0\}$ and $F_v=\{tv\colon t\ge 0\}$ and let $T$ be the tangent  to $B_2^2(0,r)\cap C$ which intersects $F_u$ in the point $tu$, $t>0$. Let $s(t)v$ be the point  where $T$ intersects $F_v$, $s(t)>0$.

	In the case $\pi/2\le \alpha<\pi$,  Lemma~\ref{lem:triangle} implies that the triangle formed by the origin, $tu$ and $s(t)v$ has height 
	\[
	r=\frac{ts(t)\sin\alpha}{\sqrt{t^2+s(t)^2-2ts(t)\cos\alpha}}.
	\]
	Solving  for $s(t)$, we get
	\begin{equation} \label{eq:s-def}
	s(t)=r\, t \frac{\sin\alpha \sqrt{t^2-r^2}-r\cos\alpha}{t^2\sin^2\alpha-r^2}
	=r h(t/r,\alpha),
	\end{equation}

	Let now $0<\alpha<\pi/2$. If $t < r\cos(\alpha)^{-1}$, then, again by Lemma~\ref{lem:triangle}, $s(t)$ is as in \eqref{eq:s-def}. If $t\ge r\cos(\alpha)^{-1}$, then $s(t)=r=rh(t/r,\alpha)$. 
\end{proof}

Now we can prove Theorem~\ref{thm:polygon}.

\begin{proof}[Proof of Theorem~\ref{thm:polygon}]

After a possible rescaling we can assume that $|\partial Q|=1$. We adapt the proof of \cite[Thm. 3]{BHB98} in order to first prove the distributional convergence
\begin{equation} \label{eq:dist-conv}
	\lim_{n\to\infty}\IP(n \delta_H(Q,Q_n)\le r)
= \prod_{i=1}^M (1-q_i(r)), \quad r>0,
\end{equation}
where
\[
q_i\colon \IR_+\to \IR_+,\quad
r\mapsto r\int_{\ell_{\alpha_i}}^{\infty}e^{-r(y+h(y,\alpha_i))}\dd y, 
\]
with $\ell_{\alpha}$ and $h(y,\alpha)$  as in Theorem~\ref{thm:polygon}. Recall that $Q_n$ is the convex hull of $n$ independently and uniformly distributed points $X_1,\dots,X_n$ on $\partial Q$.

First, we show how \eqref{eq:dist-conv} implies the convergence in expectation which is claimed in Theorem~\ref{thm:polygon}. We have the tail bound (see also the proof of (\ref{Prop1}) of  Proposition~\ref{pro:barany})
\begin{align*}
\IP(n\delta_H(Q,Q_n)> r)
&\le \sum_{i=1}^{M}\IP(n\dist(v_i,Q_n)> r)\\
&\le \sum_{i=1}^{M}\IP(\{X_1,\dots,X_n\}\cap B_2^2(v_i,r/n)=\emptyset)\\
&\le M e^{-c r}, \quad n\in\IN,
\end{align*}
where $c=2/|\partial Q|$.  Thus, for $p>1$,
\[
\sup_{n\in\IN}\IE [n\delta_H(Q,Q_n)]^p  
=\sup_{n\in\IN}\int_0^{\infty} \IP(n\delta_H(Q,Q_n)\ge r^{1/p})\dd r
\le M\int_0^{\infty} e^{-cr^{1/p}}\dd r
<\infty,
\]
i.e., the sequence $(n\delta_H(Q,Q_n))_{n\in\IN}$ is uniformly integrable and 
\[
\lim_{n\to\infty}n\IE[\delta_H(Q,Q_n)]
= \int_0^{\infty} 1-\prod_{i=1}^M (1-q_i(r))\dd r.
\]
This extends to arbitrary moments.

We show \eqref{eq:dist-conv}, following \cite{BHB98}. We refer to \cite{LP18} for more information on (Poisson) point processes. We begin with some preparation. Let
\[
L=\frac{\min_{1\le i\le M} \|v_i-v_{i+1}\|_2}{2},
\]
where we set $v_{M+1}=v_1$. Then the ball $B_2^2(v_i,L), 1\le i\le M,$ only intersects the edges of $Q$ which are incident to $v_i$.

For $1\le i\le M$, define the point process $\xi_{n,i}=\sum_{j=1}^{n}\delta_{n(X_j-v_i)\cap B_2^2(0,Ln) }$ regarded as a random element in the space $\mathcal{N}_i$ of locally finite counting measures on the boundary of a cone $C_i$ with apex at the origin and angle $\alpha_i$. When $1\le i\le M$, we can write $C_i$ as the convex hull of two rays spanned by unit vectors $a_i,b_i\in \IR^2$. We endow $\mathcal{N}_i$, $1\le i\le M$, with the vague topology and Borel $\sigma$-field.

As $n\to\infty$, we have 
\begin{equation} \label{eq:poisson}
(\xi_{n,1},\dots,\xi_{n,M})
\xrightarrow[n\to\infty]{\rm d} (\xi_{1},\dots,\xi_{M}),
\end{equation}
where $\xi_1,\dots,\xi_M$ are independent Poisson processes on $\partial C_1,\dots,\partial C_M$, respectively, and intensity measure being Lebesgue measure. Thus, for every Borel subset $B\subset \partial C_i$, the distribution of the number of points $\xi_i(B)$ of $\xi_i$ in $B$ is a Poisson random variable with parameter $|B|$, and for any choice of disjoint Borel sets $B_1,\dots,B_m\subset \partial C_i$ the random variables $\xi_i(B_1),\dots,\xi_i(B_m)$ are independent. Moreover, these are independent from $\xi_j(A_1),\dots,\xi_j(A_m)$, whenever $A_1,\dots,A_m\subset \partial C_j$ are disjoint and $i\neq j$.

For convenience of readers unfamiliar with Poisson processes, we provide a brief justification of \eqref{eq:poisson}. Choose bounded measurable sets $A_i\subset \partial C_i$ and integers $k_i\in\IN_0$ for each $i$. Then \eqref{eq:poisson} is equivalent to 
\begin{equation} \label{eq:poisson2}
	\lim_{n\to\infty}\IP(\xi_{n,1}(A_1)=k_1,\dots,\xi_{n,M}(A_M)=k_M)
=\prod_{i=1}^M\IP(\xi_{i}(A_i)=k_i),
\end{equation}
where $\xi_i(A_i)$ is Poisson distributed with parameter $|A_i|$. For $n$ large enough such that the sets $v_i+\frac{1}{n}A_i, i=1,\dots,M,$ are disjoint and part of the polygon $Q$ which has $|\partial Q|=1$, and $\sum_{i=1}^{M}k_i<n$, we compute 
\begin{align*}
&\IP(|\{j\colon X_j\in v_1+\frac{1}{n}A_1\}|=k_1,\dots,|\{j\colon X_j\in v_M+\frac{1}{n}A_M\}|=k_M)\\
&=\frac{n!}{k_1!\cdots k_M!(n-\sum_{i=1}^{M}k_i)!}\Big(\frac{|A_1|}{n}\Big)^{k_1}\cdots\Big(\frac{|A_M|}{n}\Big)^{k_M}\Big(1-\frac{\sum_{i=1}^{M}|A_i|}{n}\Big)^{n-\sum_{i=1}^{M}k_i}.
\end{align*}
Taking the limit, we arrive at
\[
\lim_{n\to\infty} \IP(\xi_{n,1}(A_1)=k_1,\dots,\xi_{n,M}(A_M)=k_M)
=\frac{1}{k_1!\cdots k_M!}|A_1|^{k_1}\cdots |A_M|^{k_M} e^{-\sum_{i=1}^{M}|A_i|}.
\]
This proves \eqref{eq:poisson2}, and thus \eqref{eq:poisson}.

We continue the proof of Theorem~\ref{thm:polygon}. By Lemma~\ref{lem:hd-vertex} we have
\begin{equation} \label{Gleich}
n \, \delta_H(Q,Q_n)
=\max_{1\le i \le M} n\,  \dist(v_i,Q_n).
\end{equation}
Provided that for all $1\le i\le M$ the point process $\xi_{n, i}$ has nonempty intersection with both rays forming $\partial C_i$,  we have
\[
n\,  \delta_H(Q,Q_n)
=\max_{1\le i \le M} n\,  \dist(v_i,Q_n)
=\max_{1\le i\le M}f_i(\xi_{n, i}), 
\]
where $f_i$ continuously maps any locally finite counting measure $\eta\in \mathcal{N}_i$ to the smallest distance between the origin and the convex hull of the points of $\eta$. In fact, the probability that for some $1\le i\le M$ the process $\xi_{n, i}$ does not intersect one of the rays forming $\partial C_i$ is given by
\[
\IP(\exists i\colon X_j\not\in [v_i,v_i+La_i]\text{ for all } j=1,\dots,n \lor X_j\not\in [v_i,v_i+Lb_i]\text{ for all } j=1,\dots,n), 
\]
which,  by a union bound,  is bounded by $2M (1-L)^n$ as $|\partial Q|=1$. Therefore, 
\begin{equation} \label{eq:equivalence}
\IP(n\,  \delta_H(Q,Q_n)
\neq \max_{1\le i\le M}f_i(\xi_{n, i}))
\le 2M(1-L)^n \to 0.
\end{equation}
By the continuous mapping theorem it follows from \eqref{eq:poisson} that
\[
\max_{1\le i\le M}f_i(\xi_{n,i})
\xrightarrow[n\to\infty]{\rm d}\max_{1\le i\le M}f_i(\xi_{i}).
\]
Combined with \ref{Gleich}, we get
\[
n \delta_H(Q,Q_n)\xrightarrow[n\to\infty]{\rm d}\max_{1\le i\le M}f_i(\xi_{i}).
\]
If we can show that, for $i=1,\dots,M$, 
\begin{equation} \label{eq:claim}
\IP(f_i(\xi_i)>r)=q_i(r), \quad r>0,
\end{equation}
then
\[
\lim_{n\to\infty}\IP(n \delta_H(Q,Q_n)\le r)
= \IP(\max_{1\le i\le M}f_i(\xi_{i})\le r)
=\prod_{i=1}^M (1- q_i(r)),
\]
i.e., \eqref{eq:dist-conv} and thus the conclusion holds.

We show \eqref{eq:claim} for $1\le i\le M$ and  $r>0$. Define the random variable $T_i=\inf\{y>0\colon \xi_i([0,ya_i])\neq 0\}$, i.e., the distance to the origin of the closest point to the origin belonging to $\xi_i$ and being on one ray of $\partial C_i$. 

Observe that for any $R>0$, by the law of total probability and Lemma~\ref{lem:tangent},
\begin{align*}
\IP(f_i(\xi_i)>r,T_i\le R)
&=\int_0^R \IP(f_i(\xi_i)>r \vert T_i=y) \dd \IP(T_i\le y)\\
&=\int_{\ell_{\alpha_i}r}^R \IP(\xi_i([0,r h(y/r,\alpha_i)a_i])=0) \dd \IP(T_i\le y)\\
&=\int_{\ell_{\alpha_i}r}^R e^{-r h(y/r,\alpha_i)-y} \dd y,
\end{align*}
since $\IP(f_i(\xi_i)>r \vert T_i=y)$ is the probability that the points of $\xi_i$ can be separated from the origin by a line passing through $ya_i$ and having distance $r$ from the origin, conditional on the fact that one point lies in $ya_i$. Further, note that the distribution of $T_i$ has Lebesgue density 
\[
\frac{\partial}{\partial y}\IP(T_i\le y)
=\frac{\partial}{\partial y}\IP(\xi_i([0,ya_i])\neq 0)
=\frac{\partial}{\partial y}(1-\IP(\xi_i([0,ya_i])= 0))
=e^{-y}, \quad y>0.
\]
Letting $R\to\infty$ gives
\[
\IP(f_i(\xi_i)>r)
=\lim_{R\to\infty} \IP(f_i(\xi_i)>r,T_i\le R)
=\int_{\ell_{\alpha_i}r}^{\infty} e^{-r h(y/r,\alpha_i)-y} \dd y
=q_i(r).
\]
This completes the proof.
	
\end{proof}

\begin{remark}
	Note that a similar approach as in the proof of Theorem~\ref{thm:polygon} can be used to show the convergence in expectation
\[
\lim_{n\to\infty} \sqrt{n}\IE[\delta_{H}(Q,Q_{n})]
=\int_0^{\infty}1-\prod_{i=1}^{M}(1-p_{i}(r))\dd r,
\]
where $Q_n$ is the convex hull of $n$ independent random points distributed uniformly inside $Q$ and $p_i$ is as in \eqref{eq:bhb-integral}. 
\end{remark}

\begin{proof}[Proof of Corollary~\ref{cor:asymptotics}]
	With notation as in Theorem~\ref{thm:polygon}, it holds by \eqref{polygon-1} that
\begin{eqnarray*}
\lim_{n\to\infty}n \, \IE[\delta_H(Q,Q_n)]
	= \left| \partial \, Q \right| \int_{0}^{\infty}1-\prod_{i=1}^M\Big(1-r\int_{\ell_{\alpha_i}}^{\infty}e^{-r(y+h(y,\alpha_i))}\dd y\Big)\dd r.
\end{eqnarray*}
As by assumption $\alpha_i\ge \frac{\pi}{2}$ for all $i$, we have that  $\ell_{\alpha_i}=\frac{1}{\sin \alpha_i}$ for all $1\leq i\leq M$. Since the function $h \geq 0$, for all $1\leq i \leq M$,
$$
	r\int_{\ell_{\alpha_i}}^{\infty}e^{-r(y+h(y,\alpha_i))}\dd y
	\leq r\int_{\ell_{\alpha_i}}^{\infty}e^{-r\cdot y}\dd y
	=e^{-r\cdot \ell_{\alpha_i}},
$$
and hence
\begin{eqnarray*}
&&\prod_{i=1}^M\Big(1-r\int_{\ell_{\alpha_i}}^{\infty}e^{-r(y+h(y,\alpha_i))}\dd y\Big)
\geq \prod_{i=1}^M\big(1-e^{-r\cdot \ell_{\alpha_{i}}}\big)
\\
&&=1-\left(\sum_{k=1}^{M}(-1)^{k+1}\sum_{1\leq i_{1}<\cdots<i_{k}\leq M}
\exp(-r(\ell_{\alpha_{i_{1}}}+\cdots+\ell_{\alpha_{i_{k}}}))
\right).
\end{eqnarray*}
Since $\ell_{\alpha_i}=\frac{1}{\sin \alpha_i}$ for all $i=1,\dots,M$, we have
\begin{eqnarray*}
&&\lim_{n\to\infty}n \, \IE[\delta_H(Q,Q_n)]
\\
&&\leq\left| \partial \, Q \right| \int_{0}^{\infty}
\left(\sum_{k=1}^{M}(-1)^{k+1}\sum_{1\leq i_{1}<\cdots<i_{k}\leq M}
\exp(-r(\ell_{\alpha_{i_{1}}}+\cdots+\ell_{\alpha_{i_{k}}}))
\right)\dd r
\\
&&=\left| \partial \, Q \right| 
\sum_{k=1}^{M}(-1)^{k+1}\sum_{1\leq i_{1}<\cdots<i_{k}\leq M}
\frac{1}{\ell_{\alpha_{i_{1}}}+\cdots+\ell_{\alpha_{i_{k}}}}
\\
&&=\left| \partial \, Q \right|
\sum_{k=1}^{M}(-1)^{k+1}\sum_{1\leq i_{1}<\cdots <i_{k}\leq M}
\frac{1}{\frac{1}{\sin\alpha_{i_{1}}}+\cdots+\frac{1}{\sin\alpha_{i_{k}}}}.
\end{eqnarray*}
Now we treat the inverse inequality. For $y\geq\frac{2}{\sin\alpha}$ we have, for any $\alpha\ge \frac{\pi}{2}$,
$$
\frac{y^2}{2}\sin^{2}\alpha \leq y^2 \sin^{2}\alpha  -1.
$$
Therefore,
\begin{eqnarray*}
h(y,\alpha)
= y\frac{\sin(\alpha)\sqrt{y^2-1}-\cos(\alpha)}{y^2 \sin^2(\alpha)-1}
\leq y\frac{\sin(\alpha)y-\cos\alpha}{y^2 \sin^2(\alpha)-1}
\leq\frac{2}{\sin\alpha}-\frac{2\cos\alpha}{y\sin^{2}\alpha}
\leq\frac{3}{\sin\alpha}.
\end{eqnarray*}
It follows that
\[
r\int_{\ell_{\alpha}}^{\infty}e^{-r(y+h(y,\alpha))}\dd y
\geq r\int_{2\ell_{\alpha}}^{\infty}e^{-r(y+\frac{3}{\sin\alpha})}\dd y
=re^{-r\frac{3}{\sin\alpha}}\int_{2\ell_{\alpha}}^{\infty}e^{-ry}\dd y
=e^{-r\frac{5}{\sin\alpha}}.
\]
We complete the proof with 
\begin{eqnarray*}
&&\lim_{n\to\infty}n \, \IE[\delta_H(Q,Q_n)]
= \left| \partial \, Q \right| \int_{0}^{\infty}1-\prod_{i=1}^M\big(1-r\int_{\ell_{\alpha_i}}^{\infty}e^{-r(y+h(y,\alpha_i))}\dd y\big)\dd r
\\
&&\geq\left| \partial \, Q \right| \int_{0}^{\infty}1-\prod_{i=1}^M
\big(1-e^{-r\frac{5}{\sin\alpha_{i}}}\big)\dd r
\\
&&=\left| \partial \, Q \right| \int_{0}^{\infty}\left(\sum_{k=1}^{M}(-1)^{k+1}\sum_{1\leq i_{1}<\cdots<i_{k}\leq M}
\exp\left(-5r\left(\frac{1}{\sin\alpha_{i_{1}}}+\cdots+\frac{1}{\sin\alpha_{i_{k}}}\right)\right)
\right)\dd r
\\
&&=\frac{1}{5}\left| \partial \, Q \right| \left(\sum_{k=1}^{M}(-1)^{k+1}\sum_{1\leq i_{1}<\cdots<i_{k}\leq M}
\frac{1}{\frac{1}{\sin\alpha_{i_{1}}}+\cdots+\frac{1}{\sin\alpha_{i_{k}}}}
\right).
\end{eqnarray*}
\end{proof}

\begin{proof}[Proof of Proposition~\ref{pro:smooth}]
Define for $n\in\IN$ the random variable 
\[
X_n=\Big(\frac{n}{\log n}\Big)^{2/(d-1)}\delta_H(K,K_n).
\]
We deduce from tail bounds given in \cite{DW96} that $X_n\xrightarrow{\IP} c_K$ implies $\IE[X_n] \to c_K$.

As $K$ is $C^3_+$, there is in particular at each boundary point $x \in \partial K$ a unique normal $N(x)$ and there exists $\ell>0$ such that 
$|N(x)-N(y)| \leq \ell |x-y|$ for all $x,y \in \partial K$, see \cite {DW96}.
Then, for $w>0$ and $n\in\IN$,
\[
\IP(X_n> \ell \, w )
=\IP\Big(\delta_H(K,K_n)> \ell \, w \Big(\frac{\log n}{n}\Big)^{2/(d-1)}\Big)
=\IP(\delta_H(K,K_n)> \ell \delta_n^2),
\]
where we set $\delta_n=\sqrt{w} \Big(\frac{\log n}{n}\Big)^{1/(d-1)}$.  Thus, by \cite[Thm. 1(b)]{DW96}, we have
\[
\IP(X_n> \ell \, w )
\le \IP(\partial K\not\subset B(Z_n,\delta_n)),
\]
where $Z_n=\{X_1,\dots,X_n\}$ and $B(Z_n,\delta_n)=\bigcup_{i=1}^n B_2^2(X_i,\delta_n)$ is a union of random Euclidean balls. From the proof of \cite[Lem. 1]{DW96}, we get
\[
\IP(\partial K\not\subset B(Z_n,\delta_n))
\le (4R+\alpha)^d (\delta_n/2)^{-d}\exp(-\alpha n (\delta_n/2)^{d-1})
\]
for some constants $R,\alpha>0$ and both,  $w$ and $n$ sufficiently large. This yields $c,C>0$ independent of $w$ and $n$ such that
\begin{equation} \label{eq:tail-Xn}
\IP(X_n\ge \ell \, w )
\le C w^{-d/2} \Big(\frac{n}{\log n}\Big)^{d/(d-1)} \exp(-c w^{(d-1)/2} \log n).
\end{equation}

Pick $S>\max\{c_K,\ell\}$. Then
\begin{equation} \label{eq:decomposition}
\IE [X_n]
=\int_0^{\infty} \IP(X_n\ge w) \dd w
=\int_0^{S} \IP(X_n\ge  w) \dd w
+\ell^{-1}\int_{S\, \ell^{-1}}^{\infty} \IP(X_n> \ell \, w ) \dd w.
\end{equation}
Since $X_n\xrightarrow{\IP} c_K$, by the Dominated Convergence theorem, the first summand satisfies
\[
\int_0^{S} \IP(X_n\ge  w) \dd w
\to \int_0^{S} \IP(c_K\ge  w) \dd w
=c_K.
\]
Using \eqref{eq:tail-Xn}, we will show that the second summand in \eqref{eq:decomposition} tends to zero. It satisfies
\[
\ell^{-1}\int_{S\, \ell^{-1}}^{\infty} \IP(X_n> \ell \, w) \dd w
\le C  \Big(\frac{n}{\log n}\Big)^{d/(d-1)} \ell^{-1}\int_{S\, \ell^{-1}}^{\infty} w^{-d/2} \exp(-c w^{(d-1)/2} \log n)  \dd w.
\]
Since $L:=S\, \ell^{-1}> 1$, we estimate the integral from above using $w^{-d/2}\le 1$ for $w\ge L$. Then it remains to bound, for $\alpha=\frac{d-1}{2}$ and $c_n=c\log n$,
\[
\int_{L}^{\infty} e^{-c_n w^\alpha }  \dd w
= c_n^{-1/\alpha} \int_{c_n^{1/\alpha}L}^{\infty} e^{-w^\alpha}  \dd w.
\]
The tail bound
\[
\int_t^{\infty} e^{-u^{\alpha}}\dd u 
\le \int_t^{\infty} \frac{u^{\alpha-1}}{t^{\alpha-1}}e^{-u^\alpha}\dd u 
= \frac{e^{-t^\alpha}}{\alpha t^{\alpha -1}},\quad t>0,
\]
gives
\[
\int_{L}^{\infty} e^{-c_n w^\alpha }  \dd w
\lesssim  n^{-cL^{\alpha}},
\]
which goes to zero faster than any fixed polynomial if $S$ is chosen large enough. Thus, the second summand in \eqref{eq:decomposition} tends to zero and the proof is complete.

\end{proof}

\begin{proof}[Proof of Proposition~\ref{pro:barany}]
Let $v_1,\dots,v_M$ be the vertices of $Q$ and $Q_n=[X_1,\dots,X_n]$, where $X_1,\dots,X_n$ are sampled independently at random from $\partial Q$. Then Lemma~\ref{lem:hd-vertex} implies
\[
n^{1/(d-1)}\IE[\delta_H(Q,Q_n)]
=\int_0^{\infty}\IP\big(\max_{1\le i\le M}\dist(v_i,Q_n)\ge rn^{-1/(d-1)}\big)\dd r,\quad n\in\IN.
\]
Let $n\in\IN$ and $r>0$. If $rn^{-1/(d-1)}\ge \diam(Q)$ the integrand vanishes. Otherwise, a union bound gives
\[
\IP\big(\max_{1\le i\le M}\dist(v_i,Q_n)\ge rn^{-1/(d-1)}\big)
\le \sum_{i=1}^{M}\IP(\dist(v_i,Q_n)\ge rn^{-1/(d-1)}).
\]
For all $i=1,\dots,M$, if $\dist(v_i,Q_n)\ge rn^{-1/(d-1)}$, then 
  \[
    B^2_2(v_i,rn^{-1/(d-1)})\cap \{X_1,\dots,X_n\}=\emptyset.
  \] 
	Further, there exists a constant $c=c(Q)\in (0,\infty)$ such that, for all $i=1,\dots,M$ and for all $\varrho\in (0,\diam(Q)]$, 
\begin{equation} \label{eq:volume-ball}
\IP(X_1\in B^2_2(v_i,\varrho))\ge c\, \varrho^{d-1}.
\end{equation}
Therefore, for all $i=1,\dots,M$, 
\begin{equation}\label{Prop1}
\IP\big(\dist(v_{i},Q_n)\ge rn^{-1/(d-1)}\big)\le (1-cr^{d-1}/n)^n\le e^{-cr^{d-1}}.
\end{equation}
Summing over $i=1,\dots,M$ and integrating over $r\in (0,\infty)$ gives an upper bound, independent of $n$, on the quantity $n^{1/(d-1)}\IE[ \delta_H(Q,Q_n)]$.

For the lower bound we note that for every $i=1,\dots,M$ there exists $c_i\in (0,\infty)$ such that for all $\varrho>0$ and all point sets $\{x_1,\dots,x_n\}\subset \partial Q$, we have
\begin{equation} \label{eq:dist-empty}
	B^2_2(v_i,c_i \varrho)\cap \{x_1,\dots,x_n\}=\emptyset \quad \Rightarrow \quad \dist(v_i,[x_1,\dots,x_n])\ge \varrho.
\end{equation}
Before we prove this relation, we use it to complete the proof of the lower bound. Let $C=\max_{1\le i\le M}c_i$ and let $R>0$ be small enough such that the balls $B^2_2(v_i,C\, R), 1\le i\le M,$ are pairwise disjoint. Let $r\in (0,R]$. Then \eqref{eq:dist-empty} implies that, for all $n\in\IN$,
\begin{align*}
\IP(\delta_H(Q,Q_n)\ge rn^{-1/(d-1)})
&\ge \IP(\{X_1,\dots,X_n\}\cap\bigcup_{i=1}^M B^2_2(v_i,C \, r\, n^{-1/(d-1)})=\emptyset)\\
&= (1-c\, C^{d-1}r^{d-1}/n)^n,
\end{align*}
which is at least $\frac{1}{2}$ provided that $R$ is small enough. 
Then we have,  for all $n\in\IN$,
\[
n^{1/(d-1)}\IE[\delta_H(Q,Q_n)]
\ge \int_0^R \IP(\delta_H(Q,Q_n)\ge rn^{-1/(d-1)})\dd r
\ge R/2.
\]
This  completes the lower bound since $R$ depends only on $C$ and thus only on $Q$.

It remains to show \eqref{eq:dist-empty}. Fix $1\le i\le M$ and $\varrho>0$. We will find $c_i>0$ such that for any point set $\{x_1,\dots,x_n\}\subset \partial Q$ with $B^2_2(v_i,\varrho)\cap \{x_1,\dots,x_n\}=\emptyset$, we have 
\[
\dist(v_i,[x_1,\dots,x_n])\ge c_i^{-1}\varrho.
\]
The distance $\dist(v_i,[x_1,\dots,x_n])$ becomes minimal if $A=\{x_1,\dots,x_n\}\cap \partial B^2_2(v_i,\varrho)$ contains at least two points, say $x$ and $y$. Then $\dist(v_i,[x_1,\dots,x_n])$ is equal to the minimal distance between $v_i$ and the line connecting $x$ and $y$. A brief calculation using Lemma~\ref{lem:triangle} shows that, for any $x,y\in A$, this distance is equal to 
\begin{equation} \label{eq:height-2}
	\rho \, \frac{\sqrt{1+\cos(\alpha_i)}}{\sqrt{2}},\quad\text{where}\quad\alpha_i=\angle(x,v_i,y)\in (0,\pi),
\end{equation}
thus concluding the proof.
\end{proof}

\section*{Acknowledgments}

JP is supported by the German Research Foundation (DFG) under project 516672205 and by the Austrian Science Fund (FWF) under project P-32405.

\par
\noindent
Part of this work has been carried out during MS's stay at the Department of Mathematics of CWRU whose hospitality is gratefully acknowledged.  MS's stay was supported by the State of Upper Austria.  MS is supported by the Austrian Science Fund (FWF) under project P-32405. This research was funded in whole or in part by the Austrian Science Fund (FWF) [Grant DOI: 10.55776/P32405; 10.55776/J4777]. 
\par
\noindent
EMW is  supported by NSF grant DMS-2103482.
\par
\noindent
Further, we thank the Hausdorff Research Institute for Mathematics, University of Bonn, for providing an excellent  working environment during the Dual Trimester Program \emph{Synergies between modern probability, geometric analysis and stochastic geometry}.

\bibliographystyle{plain}
\bibliography{convexapp}
\end{document}